\documentclass[11pt,a4paper,twoside]{article}
\usepackage{a4wide}

\usepackage[utf8]{inputenc}
\usepackage[T1]{fontenc}
\usepackage{amsmath,amsthm,amsfonts,bbm}
\usepackage{graphicx}

\newtheorem{theorem}{Theorem}

\newcommand{\R}{\mathbb{R}}
\newcommand{\C}{\mathbb{C}}
\newcommand{\N}{\mathbb{N}}
\let\vec\boldsymbol

\begin{document}

\title{A Space-Time Continuous Galerkin Finite Element Method for Linear Schrödinger Equations}

\author{Marco Zank}
\date{
      Institut f\"ur Analysis und Scientific Computing, TU Wien, Vienna, Austria \\[1mm]
      {\tt marco.zank@tuwien.ac.at}
      }

\maketitle

\begin{abstract}
    We introduce a space-time finite element method for the linear time-de\-pen\-dent Schrö\-dinger equation with Dirichlet conditions in a bounded Lipschitz domain. The proposed discretization scheme is based on a space-time variational formulation of the time-de\-pen\-dent Schrö\-dinger equation. In particular, the space-time method is conforming and is of Galerkin-type, i.e., trial and test spaces are equal. We consider a tensor-product approach with respect to time and space, using piecewise polynomial, continuous trial and test functions. In this case, we state the global linear system and efficient direct space-time solvers based on exploiting the Kronecker structure of the global system matrix. This leads to the Bartels--Stewart method and the fast diagonalization method. Both methods result in solving a sequence of spatial subproblems. In particular, the fast diagonalization method allows for solving the spatial subproblems in parallel, i.e., a time parallelization is possible. Numerical examples for a two-dimensional spatial domain illustrate convergence in space-time norms and show the potential of the proposed solvers.

    \bigskip

\end{abstract}

\section{Introduction}

The Schrö\-ding\-er equation is the central governing equation in quantum mechanics, where its solution, the wave function, is usually complex-valued and yields the time evolution of the state in quantum mechanics. Moreover, in several physical fields, the Schrö\-ding\-er equation plays a major role. Classical approaches for computing approximate solutions for Schrö\-ding\-er equations separate the temporal and spatial directions. In particular, time-stepping schemes are applied in connection with different spatial discretizations, e.g., finite difference methods, finite element methods or spectral methods. We refer to the overview in the review article \cite{LasserLubich2020}. An alternative is an approach via space-time methods, where the temporal variable $t$ is just another spatial variable. Recently, space-time methods for the linear Schrö\-ding\-er equation have been introduced. In particular, nonconforming approaches of discontinuous Galerkin-type are considered in \cite{Demkowicz2017} and \cite{GomezMoiola2024}. In \cite{BressanKushovaSangalliTani2023}, a least squares approach based on \cite{Demkowicz2017} is considered, whereas \cite{HainUrbanSchroedinger2024} relies on an ultraweak variational formulation. However, these approaches in \cite{BressanKushovaSangalliTani2023} and \cite{HainUrbanSchroedinger2024} require $H^2$ conformity, which could lead to a delicate construction of the approximation spaces.

In this work, we introduce a space-time finite element method for the time-de\-pen\-dent Schrö\-ding\-er equation, which only requires $H^1$ conformity. We apply a tensor-product approach using piecewise polynomial, continuous trial and test functions. We state the resulting global linear system. Further, we investigate efficient direct space-time solvers, including time parallelization, which are based on the Bartels--Stewart method~\cite{BartelsStewart1972} or on the fast diagonalization method~\cite{Lynch1964}. The model problem is the nondimensionalized time-de\-pen\-dent Schrö\-ding\-er equation with Dirichlet conditions to find the wave function $\psi \colon \, \overline{\Omega} \times [0,T] \to \C$ such that
\begin{equation} \label{zank:SchroedingerBeschraenkt}
    \left. \begin{array}{rclcl}
    \mathrm{i} \partial_t \psi(x,t) - \Delta_x \psi(x,t) & = & f(x,t) & \quad & \mbox{for} \;
    (x,t) \in Q = \Omega \times (0,T), \\[1mm]
    \psi(x,t) & = & 0 & & \mbox{for} \; (x,t) \in  \partial\Omega \times [0,T],
    \\[1mm]
    \psi(x,0) & = &\psi_0(x) & & \mbox{for} \; x \in \Omega,
    \end{array} \right\}
\end{equation}
where $f\colon \, \Omega \times (0,T) \to \C$ is a given right-hand side, and $\psi_0\colon \, \Omega \to \C$ is a given initial condition. Further, $\Omega \subset \R^d$, $d \in \N$, is a bounded Lipschitz domain, which is an interval $\Omega = (0,L)$ for $d=1$, polygonal for $d=2$, or polyhedral for $d=3$, and $T>0$ is the terminal time.

The rest of the paper is organized as follows: In Section~\ref{zank:sec:framework}, the used function spaces and the unique solvability of the Schrö\-ding\-er equation are considered. Section~\ref{zank:sec:FEM} is devoted to the introduction of the space-time finite element method and space-time solvers. In Section~\ref{zank:sec:Num}, we present numerical examples in a two-dimensional spatial domain. Finally, we draw some conclusions in Section~\ref{zank:sec:Zum}.

\section{Functional Framework} \label{zank:sec:framework}

In this section, we introduce function spaces and state unique solvability for the time-de\-pen\-dent Schrö\-ding\-er equation~(\ref{zank:SchroedingerBeschraenkt}). For this purpose, we define the solution space
\[
    W(Q) = \{ u \in L^2(0,T;H^1_0(\Omega)) : \partial_t u \in L^2(0,T; [H^1_0(\Omega)]^*) \}.
\]
As $W(Q) \subset C([0,T]; L^2(\Omega))$, we introduce the subspace
\[
  W_{0,}(Q) = \{ u \in W(Q): u(\cdot,0)=0 \}.
\]
Here, all spaces are complex vector spaces. In particular, the usual Lebesgue space $L^2(\mathrm{D})$ consists of complex-valued functions $u\colon\, \mathrm{D} \to \C$ and is endowed with the inner product
\[
  \langle u, v \rangle_{L^2(\mathrm{D})} = \int_\mathrm{D} u(x) \overline{v(x)} \mathrm dx,
\]
where $\mathrm{D} \subset \R^m$, $m \in \N$, is a bounded Lipschitz domain. Analogously, we consider a vector-valued version $L^2(\mathrm{D})^m$ and the extension to Bochner spaces $L^2(0,T;X)$ for a complex Hilbert space $X$. Moreover, the Sobolev space $H^1_0(\mathrm{D})$ consists of complex-valued functions, which are zero on the boundary $\partial \mathrm{D}$, and we equip $H^1_0(\mathrm{D})$ with the inner product $\langle u, v \rangle_{H^1_0(\mathrm{D})}= \langle \nabla_x u, \nabla_x v \rangle_{L^2(\mathrm{D})^m}$, inducing a norm due to the Poincar\'{e} inequality. The antidual $[H^1_0(\mathrm{D})]^*$ is the space of conjugate linear continuous functionals from $H^1_0(\mathrm{D})$ to $\C$, where $\langle \cdot, \cdot \rangle_\mathrm{D}$ extends continuously the $L^2(\mathrm{D})$ inner product to functionals. Last, the space $C([0,T]; X)$ contains all continuous functions $u \colon \, [0,T] \to X$ for a complex Hilbert space $X$. With this notation, we state the unique solvability of the time-de\-pen\-dent Schrö\-ding\-er equation~(\ref{zank:SchroedingerBeschraenkt}).
\begin{theorem} \label{zank:thm:loesbarkeit}
    Let the given right-hand side $f \in L^2(Q)$ satisfy the condition
    \[
      \partial_t f \in L^2(0,T; [H^1_0(\Omega)]^*)
    \]
    and let the given initial data $\psi_0 \in H^1_0(\Omega)$. Then, a unique solution $\psi \in W(Q)$ to the time-de\-pen\-dent Schrö\-ding\-er equation~(\ref{zank:SchroedingerBeschraenkt}) exists such that the regularity results
    \[
      \psi \in C([0,T]; H^1_0(\Omega)) \quad \text{ and } \quad \partial_t \psi \in C([0,T]; [H^1_0(\Omega)]^*)
    \]
    hold true.
\end{theorem}
\begin{proof}
  The unique solvability is proven in \cite[Theorem~1, Section~7 in Chapter~XVIII]{DautrayLions51992}. The regularity results are given in \cite[Remark~3, Section~7 in Chapter~XVIII]{DautrayLions51992}.
\end{proof}

\section{Space-Time Finite Element Method} \label{zank:sec:FEM}

In this section, we design a space-time finite element method for the time-de\-pen\-dent Schrö\-ding\-er equation~(\ref{zank:SchroedingerBeschraenkt}), which are based on a space-time variational setting stated first. The solution $\psi \in W(Q)$ of Theorem~\ref{zank:thm:loesbarkeit} satisfies the space-time variational formulation to find $\psi \in W(Q)$ such that $\psi(\cdot,0)=\psi_0$ and
\begin{equation} \label{zank:VF_inhomogen}
    \forall v \in L^2(0,T;H^1_0(\Omega)): \quad a(\psi,v) = \langle f, v \rangle_{L^2(Q)}
\end{equation}
with the sesquilinear form $a(\cdot,\cdot) \colon \, W(Q) \times L^2(0,T;H^1_0(\Omega)) \to \C$,
\[
    a(u,v) =  \mathrm{i} \langle \partial_t u, v \rangle_{Q} + \langle \nabla_x u, \nabla_x v \rangle_{L^2(Q)^d}.
\]
Further, we introduce the function
\[
  \widehat \psi(x,t) = \psi(x,t) - \psi_0(x) \quad \mbox{ for } (x,t) \in Q,
\]
which satisfies $\widehat \psi \in W_{0,}(Q)$. Plugging this into the variational formulation~(\ref{zank:VF_inhomogen}) yields the space-time variational formulation to find $\widehat \psi \in W_{0,}(Q)$ such that
\[
    \forall v \in L^2(0,T;H^1_0(\Omega)): \quad a(\widehat \psi,v) = \langle f, v \rangle_{L^2(Q)} - a(\psi_0,v).
\]
Next, we introduce the approximation spaces used for the space-time finite element method. For this purpose, we decompose the spatial domain $\Omega$ by
$
 \overline{\Omega} = \bigcup_{\ell=1}^{N_x} \overline{\omega}_\ell
$
with $N_x \in \N$ spatial elements $\omega_\ell \subset \R^d$, where we assume that the spatial mesh
\[
  \mathcal T = \{ \omega_\ell \}_{\ell=1}^{N_x}
\]
is admissible. In the following, a shape-regular sequence $(\mathcal T_\nu)_{\nu \in \N}$ of decompositions of $\Omega$ is considered, which, for simplicity, consists of intervals for $d=1$, triangles for $d=2$, and tetrahedra for $d=3$ as elements $\omega_\ell$. The temporal mesh is given by the vertices
\[
  0 = t_0 < t_1 < \dots < t_{N_t}=T
\]
for $N_t \in \N$ elements $(t_{\ell-1}, t_{\ell}) \subset \R$. We denote by $h_x$ and $h_t$ the maximal mesh size with respect to space and time, respectively. To the temporal mesh and the spatial mesh $\mathcal T_\nu$, we relate the space-time finite element space
\[
  Q^p_h(Q) = S^p_{h_x}(\Omega) \otimes S^p_{h_t}(0,T)
\]
for a given polynomial degree $p \in \N$. Here, $S^p_{h_x}(\Omega)$ is the space of piecewise polynomial, continuous and complex-valued functions on intervals for $d=1$, or triangles for $d=2$, or tetrahedra for $d=3$ such that $v_{|\overline{\omega}}$ is a polynomial function of degree $p$ for all elements $\omega \in \mathcal T_\nu$, where $v \in  S^p_{h_x}(\Omega)$. Analogously, we define $S^p_{h_t}(0,T)$ as the space of piecewise polynomial, continuous and complex-valued functions. Last, we introduce the subspace $Q^p_{h,0,}(Q) = Q^p_h(Q) \cap W_{0,}(Q)$
of complex-valued functions, which satisfy the homogeneous initial and boundary conditions. With this, we consider the conforming space-time finite element method to find $\widehat \psi_h \in Q^p_{h,0,}(Q)$ such that
\begin{equation} \label{zank:VF_FEM}
    \forall v_h \in Q^p_{h,0,}(Q): \quad a(\widehat \psi_h,v_h) = \langle f, v_h \rangle_{L^2(Q)} - a(\psi_0, v_h).
\end{equation}
Finally, we define the approximation $\psi \approx \psi_h$ by
\[
  \psi_h(x,t) = \widehat \psi_h(x,t) + \psi_0(x) \quad \mbox{ for } (x,t) \in Q.
\]

The numerical analysis of the space-time finite element method~(\ref{zank:VF_FEM}) is not in the scope of this work and will be investigated in future work. Anyway, in this work, we focus on space-time solvers for the space-time finite element method~(\ref{zank:VF_FEM}). For this purpose, we state the resulting global linear system
\begin{equation} \label{zank:LGS}
    (\mathrm{i} B_t \otimes M_x + M_t \otimes A_x) \vec \psi = \vec f.
\end{equation}
Here, the real-valued, spatial matrices $M_x, A_x \in \R^{n_x \times n_x}$ are defined by
\begin{equation} \label{zank:matrix_x}
  M_x[k,j] = \int_\Omega \phi_j(x) \phi_k(x) \mathrm dx \quad \mbox{ and } \quad A_x[k,j] = \int_\Omega \nabla_x \phi_j(x) \cdot \nabla_x \phi_k(x) \mathrm dx
\end{equation}
for $k,j=1,\dots,n_x$ with real-valued basis functions $\phi_k \colon \overline{\Omega} \to \R$. The real-valued, temporal matrices $M_t, B_t \in \R^{n_t \times n_t}$ are
\begin{equation} \label{zank:matrix_t}
  M_t[k,j] = \int_0^T \varphi_j(t) \varphi_k(t) \mathrm dt \quad \mbox{ and } \quad B_t[k,j] = \int_0^T \partial_t \varphi_j(t) \varphi_k(t) \mathrm dt
\end{equation}
for $k,j=1,\dots,n_t$ with real-valued basis functions $\varphi_k \colon [0,T] \to \R$. The right-hand side $\vec f \in \C^{n_x \cdot n_t}$ of the linear system~(\ref{zank:LGS}) is given analogously. Note that the number of spatial degrees of freedom is $n_x$, whereas the number of temporal degrees of freedom is $n_t$, which leads to the total number of degrees of freedom $n = n_x \cdot n_t.$

Next, we state efficient space-time solvers for the global linear system~(\ref{zank:LGS}). To develop such space-time solvers, which exploit the Kronecker structure of the global linear system~(\ref{zank:LGS}), we assume the temporal decomposition
\[
    (\mathrm{i} B_t)^{-1} M_t = X_t S_t X_t^{-1}.
\]
Here, the matrix $X_t$ is regular, and the matrix $S_t \in \C^{n_t \times n_t}$ is regular and upper triangular, where $X_t, S_t$ have to be specified. Further, we set
\[
  Y_t = X_t^{-1} (\mathrm{i} B_t)^{-1}.
\]
With these representations, the solution $\vec \psi \in \C^{n_x \cdot n_t}$ of the global linear system~(\ref{zank:LGS}) is given by
\[
  \vec \psi = (X_t \otimes I_{n_x}) (I_{n_t} \otimes M_x + S_t \otimes A_x)^{-1} (Y_t \otimes I_{n_x}) \vec f,
\]
see \cite[Section~3.2]{FoltynLukasZankPRESB2024} and \cite[Eq.~(4.2)]{LangerZank2021}, where $I_{n_t} \in \R^{n_t \times n_t}$, $I_{n_x} \in \R^{n_x \times n_x}$ are identity matrices. This leads to the following Kronecker product solver, see \cite[Algorithm~1]{FoltynLukasZankPRESB2024}, \cite[Algorithm~4.2]{LangerZank2021} for the real-valued case:
\begin{enumerate}
  \item Compute the matrices $X_t, S_t$ of the decomposition $(\mathrm{i} B_t)^{-1} M_t = X_t S_t X_t^{-1}$.
  \item Calculate $\vec g = (\vec g_1, \dots, \vec g_{n_t})^\top = (Y_t \otimes I_{n_x}) \vec f$ with the vectors $\vec g_\ell \in \C^{n_x}$.
  \item For $\ell=n_t, n_t-1,\dots, 1$, solve the spatial problems
  \begin{equation} \label{zank:LGS_x}
    (M_x + S_t[\ell,\ell] A_x) \vec w_\ell = \vec g_\ell - \sum_{k=\ell+1}^{n_t} S_t[\ell,k] A_x \vec w_k
  \end{equation}
  for $\vec w_\ell \in \C^{n_x}$, where $\sum_{k=n_t+1}^{n_t}(\cdot) = 0$.
  \item Compute the solution $\vec \psi \in \C^{n_x \cdot n_t}$ by $\vec \psi = (X_t \otimes I_{n_x}) \vec w.$
\end{enumerate}
Note that steps 2 and 4 consists of few matrix multiplications, which are parallelizable and can be written as highly efficient operations, see \cite[Subsection~4.3.1, Subsection~4.4.1]{LangerZank2021}. In the case that $n_t \ll n_x$, the most expensive part is step~3, i.e., solving $n_t$ spatial subproblems of size $n_x$. This can be done by an iterative solver~\cite{FoltynLukasZankPRESB2024} or by sparse direct solvers~\cite{LangerZank2021}. The latter is considered in this work.

Finally, we state two choices for the complex matrices $X_t, S_t$ of the decomposition
\[
  (\mathrm{i} B_t)^{-1} M_t = X_t S_t X_t^{-1}
\]
in step 1 of the proposed solver:
\begin{enumerate}
  \item \underline{Bartels--Stewart method~\cite{BartelsStewart1972}}: The first choice is the Schur decomposition leading to a unitary matrix $X_t$ and an upper triangular matrix $S_t$ with the eigenvalues of $(\mathrm{i} B_t)^{-1} M_t$ on the diagonal of $S_t$.
  \item \underline{Fast diagonalization method~\cite{Lynch1964}}: The second choice is the eigenvalue decomposition resulting in the matrix $X_t$ consisting of eigenvectors of the matrix $(\mathrm{i} B_t)^{-1} M_t$ and the diagonal matrix $S_t$, where the diagonal is composed of the eigenvalues of $(\mathrm{i} B_t)^{-1} M_t$. Here, we have to assume that the matrix $(\mathrm{i} B_t)^{-1} M_t$ is diagonalizable. This is the case for a uniform temporal mesh size, see Table~\ref{zank:tab:FD} in Section~\ref{zank:sec:Num} and \cite[Section~3.2]{FoltynLukasZankPRESB2024}. However, for a complete analysis of these eigenvalues, we refer to future work. Anyway, the sum on the right side of the spatial linear systems~(\ref{zank:LGS_x}) is zero. Thus, the spatial linear systems~(\ref{zank:LGS_x}) can be solved in parallel, i.e., a time parallelization is possible.
\end{enumerate}

\section{Numerical Experiments} \label{zank:sec:Num}

In this section, we state numerical examples, which show the potential of the space-time finite element method~(\ref{zank:VF_FEM}) proposed in this work. We consider the unit square $\Omega = (0,1) \times (0,1) \subset \R^2$ as spatial domain and $T=5$. The initial spatial mesh is stated in Fig.~\ref{zank:Fig:Netze}, where we apply a uniform refinement strategy such that we have $N_x = 2048 \cdot 4^J$ spatial elements for the refinement levels $J= 0,\dots,4$.
\begin{figure}
\includegraphics[scale=0.66]{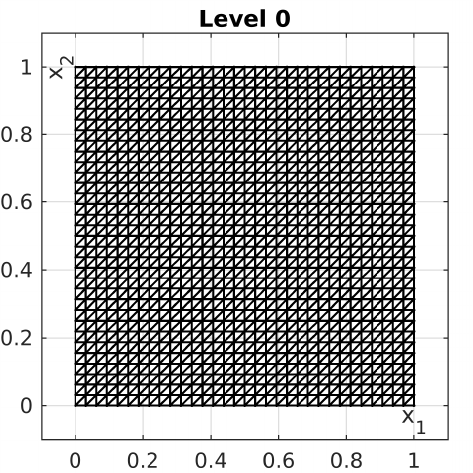}
\includegraphics[scale=0.5]{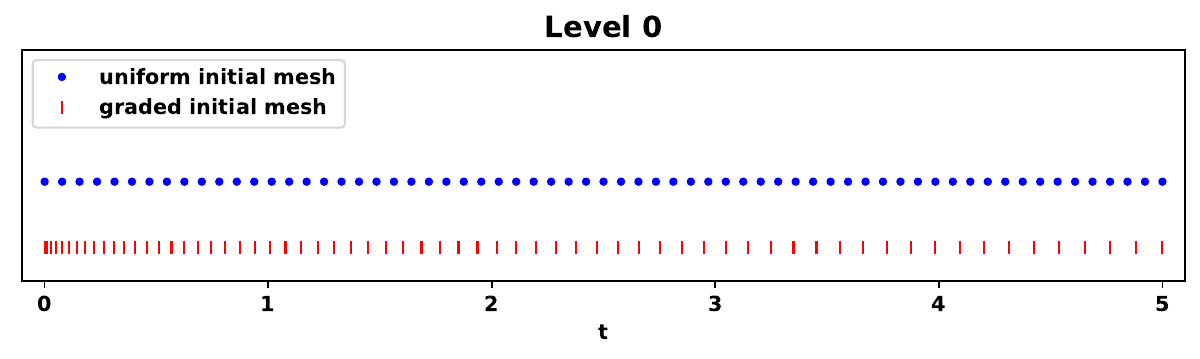}
\caption{Initial spatial mesh for $\Omega = (0,1) \times (0,1)$ and initial temporal meshes for $T=5$.} \label{zank:Fig:Netze}
\end{figure}
The uniform temporal mesh is given by
\begin{equation} \label{zank:Mesh_t}
  t_\ell = \frac{T \ell}{N_t} \quad \mbox{ for } \ell=0,\dots, N_t
\end{equation}
with $N_t = 64\cdot 2^J$ for $J=0,\dots,4$, see Fig.~\ref{zank:Fig:Netze}. The wave function $\psi \colon \, \overline{Q} \to \C$,
\[
  \psi(x_1,x_2,t) = \mathrm{e}^{\mathrm{i} t} \sin (\pi  x_1) \sin (\pi  x_2) \sin (t x_1 x_2), \quad (x_1,x_2,t) \in \overline{Q},
\]
is considered as exact solution to the time-de\-pen\-dent Schrö\-ding\-er equation~(\ref{zank:SchroedingerBeschraenkt}) with the corresponding right-hand side $f$ and the initial condition $\psi(x_1,x_2,0)= \psi_0(x_1,x_2) = 0$. In other words, we have that $\widehat \psi = \psi$ and $\widehat \psi_h = \psi_h$. We consider the polynomial degree $p=1$, i.e., the approximate solution $\psi_h \in Q^p_{h,0,}(Q)$ given by the space-time finite element method~(\ref{zank:VF_FEM}) is a piecewise multilinear function. We choose the usual nodal basis functions, i.e., the hat functions with respect to time, to compute analytically the matrices $M_x, A_x$ in (\ref{zank:matrix_x}) and the matrices $M_t, B_t$ in (\ref{zank:matrix_t}). To calculate the vector of the right-hand side~$\vec f$ of the linear system~(\ref{zank:LGS}), we apply high-order numerical integration. The numerical experiments were performed in MATLAB R2022a using internal MATLAB routines on a PC with two Intel Xeon CPUs E5-2687W v4 @ 3.00GHz, having 512 GB memory and 24 cores. In particular, we use the MATLAB routine mldivide (backslash operator) for solving the spatial problems~(\ref{zank:LGS_x}) by a sparse direct solver.

In Table~\ref{zank:tab:BSc}, we state the errors in space-time norms $\| \cdot \|_{L^2(Q)}$ and
\[
  |\cdot|_{H^1(Q)} = \left( \| \partial_t (\cdot) \|_{L^2(Q)}^2 + \| \nabla_x (\cdot) \|_{L^2(Q)^d}^2 \right)^{1/2}
\]
and the solving time in seconds s, which includes all steps of the Bartels--Stewart method but excludes the assembling time of the matrices $M_x, A_x, M_t, B_t$ and the vector $\vec f$. Here, we use a uniform refinement strategy for the spatial and temporal meshes, i.e., we quadruple the number of spatial elements $N_x$ and double the number of temporal elements $N_t$. Thus, the total number of degrees of freedom $n = n_x \cdot n_t$ grows by a factor 8. We observe second-order convergence in $\| \cdot \|_{L^2(Q)}$ and first-order convergence in $|\cdot|_{H^1(Q)}$.

\begin{table}
\caption{Results of the finite element method~(\ref{zank:VF_FEM}) with $p=1$, the temporal mesh~(\ref{zank:Mesh_t}) for $T=5$ and the initial spatial mesh in Fig.~\ref{zank:Fig:Netze} using the Bartels--Stewart method.}\label{zank:tab:BSc}
\begin{tabular}{|r|c|c|c|c|c|c|r|}
	\hline
  $n$ & $h_x$ & $h_t$ & $\|\psi - \psi_h\|_{L^2(Q)}$  & eoc & $|\psi - \psi_h|_{H^1(Q)}$ & eoc & Solve in s\\
		\hline
       61504 & 2.2e-02 & 7.8e-02 & 3.2e-03 &  -  & 2.4e-01 &  -   &     0.3 \\
      508032 & 1.1e-02 & 3.9e-02 & 8.1e-04 & 1.9 & 1.2e-01 & 1.0  &     3.1 \\
     4129024 & 5.5e-03 & 2.0e-02 & 2.0e-04 & 2.0 & 6.0e-02 & 1.0  &    30.9 \\
    33292800 & 2.8e-03 & 9.8e-03 & 5.1e-05 & 2.0 & 3.0e-02 & 1.0  &   405.4 \\
   267387904 & 1.4e-03 & 4.9e-03 & 1.3e-05 & 2.0 & 1.5e-02 & 1.0  &  4869.8 \\
    \hline
\end{tabular}
\end{table}

In Table~\ref{zank:tab:FD}, we consider the same situation as in Table~\ref{zank:tab:BSc}. However, we apply the fast diagonalization method for solving the global linear system~(\ref{zank:LGS}). Moreover, we solve the spatial linear systems~(\ref{zank:LGS_x}) in parallel using 24 cores. Thus, there is no parallelization for the direct solver  mldivide (backslash operator) possible. We observe the identical errors as for the Bartels--Stewart method in Table~\ref{zank:tab:BSc}. Comparing the solving times in Table~\ref{zank:tab:FD} and Table~\ref{zank:tab:BSc} yields that the solving time is drastically decreased by applying the fast diagonalization method. In the last column of Table~\ref{zank:tab:FD}, we state the spectral condition number $\kappa_2(X_t)$ of the transformation matrix $X_t$ related to the eigenvectors of the matrix  $(\mathrm{i} B_t)^{-1} M_t$. We report that $\kappa_2(X_t)$ does not depend on the terminal time $T$, i.e., $\kappa_2(X_t)$ depends solely on the number $n_t$ of degrees of freedom in the case of a uniform mesh size~$h_t$. Further, we observe that the spectral condition number $\kappa_2(X_t)$ does not grow exponentially with respect to $n_t$. In other words, the fast diagonalization method is applicable also for a high number of temporal degrees of freedom.

\begin{table}
\caption{Results of the finite element method~(\ref{zank:VF_FEM}) with $p=1$, the temporal mesh~(\ref{zank:Mesh_t}) for $T=5$ and the initial spatial mesh in Fig.~\ref{zank:Fig:Netze} using the fast diagonalization method.}\label{zank:tab:FD}
\begin{small}
\begin{tabular}{|r|c|c|c|c|c|c|r|c|}
	\hline
  $n$ & $h_x$ & $h_t$ & $\|\psi - \psi_h\|_{L^2(Q)}$  & eoc & $|\psi - \psi_h|_{H^1(Q)}$ & eoc & Solve in s & $\kappa_2(X_t)$ \\
		\hline
       61504 & 2.2e-02 & 7.8e-02 & 3.2e-03 &  -  & 2.4e-01 &  -   &     0.1 & 2.3e+02 \\
      508032 & 1.1e-02 & 3.9e-02 & 8.1e-04 & 1.9 & 1.2e-01 & 1.0  &     0.3 & 7.0e+02 \\
     4129024 & 5.5e-03 & 2.0e-02 & 2.0e-04 & 2.0 & 6.0e-02 & 1.0  &     2.3 & 2.2e+03 \\
    33292800 & 2.8e-03 & 9.8e-03 & 5.1e-05 & 2.0 & 3.0e-02 & 1.0  &    20.5 & 7.1e+03 \\
   267387904 & 1.4e-03 & 4.9e-03 & 1.3e-05 & 2.0 & 1.5e-02 & 1.0  &   248.8 & 2.4e+04 \\
    \hline
\end{tabular}
\end{small}
\end{table}

Last, in Table~\ref{zank:tab:FD_nonuniform}, we consider the situation as in Table~\ref{zank:tab:FD}, except that we replace the initial mesh in $t$ with the graded mesh
\[
  \{ t_\ell = T\left(\frac{\ell}{64}\right)^q: \ell=0,\dots, 64 \}
\]
with $q=1.5$, see Fig.~\ref{zank:Fig:Netze}. This initial mesh is refined uniformly when increasing the refinement level. Note that for these temporal meshes, $h_t \approx 12 \cdot h_{t,\min}$ holds with $h_{t,\min} = \min_{\ell} |t_\ell - t_{\ell-1}|$. We report that the Bartels--Stewart method and the fast diagonalization method lead to identical errors. In particular, the fast diagonalization method is also applicable for these non-uniform temporal meshes.

\begin{table}
\caption{Results of the finite element method~(\ref{zank:VF_FEM}) with $p=1$, graded initial mesh in $t$ for $T=5$ and the initial spatial mesh in Fig.~\ref{zank:Fig:Netze} using the fast diagonalization method.}\label{zank:tab:FD_nonuniform}
\begin{small}
\begin{tabular}{|r|c|c|c|c|c|c|r|c|}
	\hline
  $n$ & $h_x$ & $h_t$ & $\|\psi - \psi_h\|_{L^2(Q)}$  & eoc & $|\psi - \psi_h|_{H^1(Q)}$ & eoc & Solve in s & $\kappa_2(X_t)$ \\
		\hline
       61504 & 2.2e-02 & 1.2e-01 & 3.2e-03 &  -  & 2.4e-01 &  -   &     0.1 & 1.1e+02  \\
      508032 & 1.1e-02 & 5.8e-02 & 8.4e-04 & 1.9 & 1.2e-01 & 1.0  &     0.3 & 3.6e+02  \\
     4129024 & 5.5e-03 & 2.9e-02 & 2.1e-04 & 2.0 & 6.1e-02 & 1.0  &     2.1 & 1.1e+03  \\
    33292800 & 2.8e-03 & 1.5e-02 & 5.4e-05 & 2.0 & 3.0e-02 & 1.0  &    20.7 & 3.6e+03  \\
   267387904 & 1.4e-03 & 7.3e-03 & 1.3e-05 & 2.0 & 1.5e-02 & 1.0  &   246.6 & 1.2e+04  \\
    \hline
\end{tabular}
\end{small}
\end{table}

\section{Conclusions} \label{zank:sec:Zum}

In this work, we introduced an $H^1$ conforming space-time finite element method for the linear time-de\-pen\-dent Schrö\-ding\-er equation, using piecewise polynomial, continuous trial and test functions. Moreover, we stated efficient direct space-time solvers leading to the Bartels--Stewart method and the fast diagonalization method. The latter allows for time parallelization. We presented numerical examples for a two-dimensional spatial domain, which show the potential of the proposed methods.

\section*{Acknowledgments}
\noindent
This research was funded in part by the Austrian Science Fund (FWF) [10.55776/P36150].

\section*{Declarations}
\noindent
The author has no competing interests to declare that are relevant to the content of this article.


\end{document}